\documentclass{article}
\usepackage{amssymb,amsmath,amsthm,hyperref,graphicx}

\newtheorem{theorem}{Theorem}[section]
\newtheorem{lemma}[theorem]{Lemma}
\newtheorem{corollary}[theorem]{Corollary}

\newtheorem{example}[theorem]{Example}
\newtheorem{remark}[theorem]{Remark}
\newtheorem{proposition}[theorem]{Proposition}

\def\BBR {{\mathbb R}^4}

\def\dstyle { \displaystyle }

\begin{document}
\title{Isoperimetric inequality, finite total Q-curvature and quasiconformal map}
\author{Yi Wang}
\date{}

\maketitle
\begin{abstract}
In this paper, we obtain the isoperimetric inequality on conformally flat manifold with finite total $Q$-curvature. This is a higher dimensional analogue of Li and Tam's result \cite{L-T} on surfaces with finite total Gaussian curvature. The main step in the proof is based on the construction of a quasiconformal map whose Jacobian is suitably bounded.
\end{abstract}

\section{Introduction}

Cohn-Vossen \cite{CV} studied the Gauss-Bonnet integral for noncompact complete surface $M^2$ with analytic
metrics. He showed that if the Gaussian curvature K is absolutely
integrable(in which case we say the manifold has finite total curvature), then
\begin{equation}\label{1.1}
\dstyle \frac{1}{2\pi}\int_{M} K dv_{M}\leq  \chi(M),
\end{equation}
where $\chi(M)$ is the Euler characteristic of M. Later, Huber \cite{Hu} extended this inequality to
metrics with much weaker regularity. Another nice result he proved is that such a surface
$M^2$ is conformally equivalent to a closed surface with finitely many points removed. The difference between
the two sides of inequality (\ref{1.1}) encodes the asymptotical behavior of the manifold at its ends.  The precise geometric interpretation is given by Finn in \cite{Finn}: suppose a noncompact complete surface has absolutely integrable Gaussian curvature, one may
represent each end conformally as $\mathbb{R}^2 \setminus K$ for some compact set $K$. Define the asymptotical
isoperimetric constant to be
$$\nu=\lim_{r\rightarrow \infty}\frac{L^2(\partial B(0,r))}{4\pi A(B(0,r))},
$$
where $B(0,r)$ is the Euclidean ball centered at origin with radius $r$, $L$ is the
length of the boundary, and $A$ is the area of the domain. As Huber \cite{Hu} showed in his paper,
there are only finitely many ends and they are conformally flat, with that Finn \cite{Finn} proved
\begin{equation}\label{1.2}
\displaystyle \chi(M)-\frac{1}{2\pi}\int_{M} K dv_{M}= \sum_{j=1}^{k}\nu_{j},
\end{equation}
where $k$ is the number of ends. This result tells us that if the surface has finite total
Gaussian curvature, there are rigid geometric and analytical consequences about it.

When $n=4$ Chang, Qing and
Yang \cite{CQY1} get a generalization of (\ref{1.1}) and (\ref{1.2}) by replacing Gaussian curvature
$K$ by Paneitz Q-curvature. The $Q$-curvature comes in naturally as a conformal invariant associated to
the fourth order Paneitz operator. In order to present their result, we recall some
definitions. In conformal geometry, when $n=4$, the Paneitz operator is defined as:
$$P_g=\Delta^2+\delta(\frac{2}{3}Rg-2 Ric)d,$$
where $\delta$ is the divergence and $d$ is the differential. $R$ is the scalar curvature of $g$, and $Ric$
is the Ricci curvature tensor. The Paneitz Q-curvature is defined as
$$Q_g=\frac{1}{12}\left\{-\Delta R +\frac{1}{4}R^2 -3|E|^2 ,\right\}
$$
where $E$ is the traceless part of $Ric$, and $|\cdot|$ is taken with respect to the metric $g$.
Under conformal change $g_{w}=e^{2w}g_0$, $P_{g_w}=e^{-4w}P_{g_0}$,
and
$Q_{g_w}$ satisfies the fourth order equation,
$$P_{g_{0}}w+2Q_{g_0}=2Q_{g_{w}}e^{4w}.$$
The invariance of $Q$-curvature in dimension $4$ is due to the Chern-Gauss-Bonnet formula for closed manifold $M$:
$$\chi(M)=\dstyle\frac{1}{4\pi^2} \int_{M}\left(\frac{|W|^2}{8}+Q\right) dv_{M},$$
where $W$ denotes the Weyl tensor. The main result proved in \cite{CQY1} is that suppose $M^4=(\mathbb{R}^4, e^{2w} |dx|^2)$ is a noncompact complete conformally flat manifold
with finite total Q-curvature $\int_{\mathbb{R}^4}|Q| e^{4w}dx<\infty$, and suppose the metric is normal(or suppose the scalar curvature $R_g$ is nonnegative at infinity), then
\begin{equation}\label{1.3}
\dstyle \frac{1}{4\pi^2}\int_{M^4} Q dv_{M}\leq  \chi(\mathbb{R}^4)=1,
\end{equation}
and \begin{equation}\label{1.4}
\displaystyle \chi(\mathbb{R}^4)-\frac{1}{4\pi^2}\int_{\mathbb{R}^4} Q dv_{M}= \sum_{j=1}^{k}\lim_{r\rightarrow
\infty }\frac{vol(\partial B_{j}(r))^{4/3}}{4(2\pi^2)^{1/3}vol(B_{j}(r))}.
\end{equation}
\noindent Here $B_{j}(r)$ means the Euclidean ball with radius $r$ at the $j$-th end.
The metric is defined to be normal if
\begin{equation}\label{w}
w(x)=\frac{1}{4\pi^2}\int_{\mathbb{R}^4}\log\frac{|y|}{|x-y|}Q(y)e^{4w(y)} dy + C.
\end{equation}
The normal condition is necessary in higher dimension, because bi-Laplacian operator $\Delta^2$ has quadratic functions in its kernel
which would give counterexample to (\ref{1.3}). The assumption of positive scalar curvature at infinity would imply that the metric is normal for $n=4$, thus for the purpose of this paper, it can replace the condition of normal metric. Another point is that the result in \cite{CQY1} is not limited to $n=4$. By the same argument as that in \cite{CQY1}, it is true for all even dimensions if the metric is assumed to be normal.
Later in \cite{CQY2}, Chang, Qing and Yang generalized the above result to locally conformally
flat manifolds with certain curvature conditions and found conformal compactification of such
manifold.

Back to the case of surface, as is indicated by formula (\ref{1.2}), such manifold has isoperimetric inequality for very large Euclidean balls. A natural question is whether it would have the same inequality for arbitrarily shaped domains, instead of only for very large Euclidean balls. The answer to this question is affirmative by the work of Li and Tam in \cite{L-T}. They use the idea of comparing the domain's boundary length with the geodesic distance and analyze the asymptotic behavior of it. However on manifolds with higher dimensions, it seems almost impossible to apply the same technique. In this paper, we circumvent this difficulty by looking at the construction of quasiconformal map to prove the isoperimetric inequality for higher dimensions with totally finite $Q$-curvature. The main result is the following:

\begin{theorem}\label{c1.2} Suppose $(M^4, g)=(\mathbb{R}^4, e^{2w}|dx|^2)$ is a noncompact complete
Riemannian manifold with normal metric and satisfies \begin{equation}
\dstyle \int_{M^4} |Q| dv_{M}< \infty,
\end{equation}
and
 \begin{equation}
\dstyle \alpha := \frac{1}{4\pi^2}\int_{M^4} Q dv_{M}< 1,
\end{equation}
then $(M,g)$ has isoperimetric inequality:
\begin{equation}\label{iso}
|\Omega|_{g}^{3/4}\leq C |\partial \Omega|_{g}
\end{equation}
for any smooth bounded domain $\Omega\subset M^4$. Here the constant $C $ depends on $M$. 
The more general Sobolev inequality is also valid with $dv_M=e^{4w}dx$,
\begin{equation}\label{sob}\dstyle \left( \int_{M^4}|f(x)|^{p^*}dv_M \right)^{1/{p^*}}\leq C \dstyle \left(\int_{M^4}
(|\nabla_g f(x)|_g)^p dv_M \right)^{1/p},
\end{equation}
where $1\leq p<4$, $p^*=\frac{4p}{4-p}$. In particular, when $p=1$, (\ref{sob}) implies (\ref{iso}).
\end{theorem}

Actually we have proved a stronger result. Our Theorem \ref{c1.2} is a direct consequence of the following:
\begin{theorem}\label{main} Consider the manifold $(M^4, g)$ described in Theorem \ref{c1.2}. Suppose $w$ is the conformal factor taking the form of (\ref{w})
then there is a $H$-quasiconformal map $f:\mathbb{R}^4\rightarrow \mathbb{R}^4$ and constant $C$ such that
\begin{equation}
C^{-1} e^{4w}\leq J_{f}(x)\leq C e^{4w} \quad \mbox{a.e.}\quad  x\in \mathbb{R}^4.
\end{equation}
Here $H$ depends on $\alpha$ and $C$ depends on the manifold $M$. 
\end{theorem}
\begin{remark}It is well known that if $\frac{1}{4\pi^2}\int_{M^4} Q dv_{M}=1$, there is counterexample. In fact, one can
think of a manifold $M$ with only one end that looks like a cylinder, then it is not hard to see $\frac{1}{4\pi^2}\int_{M^4} Q dv_{M}=1$. Such a manifold $M$ doesn't have isoperimetric inequality, not to mention the q.c. map. with property (\ref{q}), because the latter would imply the former. 
\end{remark}
\begin{remark} Another thing is that even though the model case of the manifold described in Theorem \ref{c1.2} is always the cone, and we will use it several times as our example. It should be clear that $Qe^{4w}$ could be very close to the distribution:
$$\displaystyle  \sum_{k=2}^{\infty}\frac{1}{k^2} \delta_{k},$$
where $\delta_{k}$ means Dirac distribution supported at point $x=(k,0,...,0)$.
Therefore, we can't expect the manifold to be close to a cone outside a compact domain. Nor can we control the Ricci curvature from above and below in order to get isoperimetric inequality. Finally it is reasonable that the constant $C$ in Theorem \ref{main} depends on the manifold. We will explain it in Remark \ref{Remark 3.1}.
\end{remark}

There are interesting results in the direction of finding quasiconformal
 map such that its Jacobian is comparable to the volume form weight $e^{nw}$, i.e.
\begin{equation}\label{q}
C^{-1} e^{nw(x)}\leq J_{f}(x)\leq C e^{nw(x)},
\end{equation}
for almost every $x\in \mathbb{R}^n$. Here ``almost every" is because quasiconformal map's Jacobian
may not be defined every where. A quite related problem is whether there exists a
bi-Lipschitz map between $(M,g)$=$(\mathbb{R}^n, e^{2w}|dx|^2)$ and the Euclidean space $(\mathbb{R}^n, |dx|^2)$, when
assuming the total Q-curvature is finite. It is equivalent to say whether there exists a homeomorphism
$f:\mathbb{R}^n\rightarrow \mathbb{R}^n$, such that
\begin{equation}
C^{-1}d_g(x,y)\leq |f(x)-f(y)|\leq C d_g(x,y).
\end{equation}
It turns out that a Riemannian manifold is bi-Lipschitz equivalent to Euclidean space $(\mathbb{R}^n, g_0=|dx|^2)$
if and only if $e^{nw}$ is comparable to the Jacobian of some quasiconformal map. We now list some of the results about bi-Lipschitz parametrization problem in the following: when $n=2$ Fu \cite{Fu} proved that if Gaussian curvature is totally finite and its absolute integral is very small, then there exists a bi-Lipschitz map from $M$ to $\mathbb{R}^2$. Later Bonk and Lang \cite{bonk2} proved that if Gaussian curvature is totally finite and $\frac{1}{2\pi}\int_{M}|K| dv_{M}< 1$, then there exists
a bi-Lipschitz map from $M$ to $\mathbb{R}^2$. For even dimensions $n\geq 4$, the problem was solved for small total $Q$-curvature by Bonk, Heinonen
and Saksman in \cite{bonk}. They proved there exists a q.c. map satisfying (\ref{q}) if
$$\int_{M^4} |Q|dv_{M}\leq \epsilon,$$
for some small $\epsilon$ depending on dimension. Actually their result is true for all even dimensions $n$.
In Theorem \ref{main}, I am going to answer the isoperimetric inequality question by discussing Bonk,
Heinonen and Saksman's result about q.c. maps and showing the existence of q.c. map without ``small"
condition on total curvature.

As mentioned before the quasiconformal map appeared in Theorem \ref{main} is also the bi-Lipschitz map between the two metric spaces $(M^4,g)$
and $(\mathbb{R}^4,|dx|^2)$, as a corollary of Theorem \ref{main}, we find a bi-Lipschitz parametrization of $M^4$.
\begin{theorem}\label{bilip}
Consider the manifold $(M^4, g)$ with the same condition as in Theorem \ref{main}, then $(M^4, g)$
is bi-Lipschitz equivalent to $\mathbb{R}^n$, i.e there exists a bi-Lipschitz map $f:\mathbb{R}^4\rightarrow \mathbb{R}^4$, i.e
\begin{equation}
C^{-1}d_g(x,y)\leq |f(x)-f(y)|\leq C d_g(x,y).
\end{equation}
\end{theorem}

\begin{remark}\label{r1}All the above results are not restricted to the $n=4$ case. Everything
could be changed to all even dimension $n$ as long as we assume $(M^{n}, g)=(\mathbb{R}^n, e^{2w}|dx|^2)$,
where
$$w(x)=\frac{1}{c_n}\int_{\mathbb{R}^n}\log\frac{|y|}{|x-y|}P(y)dx+C
$$
for some continuous function $P(y)$ that belongs to $L^1(\mathbb{R}^n)$ and satisfies $\int_{\mathbb{R}^n}P(x)dx<c_n$. Here the dimensional
constant $c_n$ is the constant that appears in the fundamental solution equation $\Delta^{n/2}\log\frac{1}{|x|}=c_n\delta_0(x).$
\end{remark}
In harmonic analysis, the $A_p$ weight ($p\geq 1$) is an important notion which describes
when a function $u$ could be a weight such that the associated measure $u(x)dx$ has the property that
maximal function $\textrm{M}$ of a $L^1$ function is weak $L^1$ if $p=1$, and that maximal
function of $L^p$ function is $L^p$ function if $p>1$. For definitions of $A_1$, $A_p$ weight in detail, see section 4.\\

A weight $u$ is defined to be $A_\infty$
if $u\in A_p$ for some $p\geq 1$. In order to find criteria for a function to be the Jacobian of some q.c. map, David and
Semme in \cite{D-S} defined the notion of strong $A_\infty$. One of the relations
between standard $A_p$ weight and strong $A_\infty$ weight
is that an $A_1$ weight is always strong $A_\infty$. Also it is well known that if a function
is comparable to the Jacobian of a q.c. map, then it is a strong $A_{\infty}$ weight by the argument
of Gehring. Therefore

\begin{corollary}\label{cA}Under the assumption on manifold in Remark \ref{r1}, $e^{nw}$ is a strong $A_{\infty}$
weight.
\end{corollary}

Moreover, we are going to show that if Q-curvature is nonnegative(resp. P(x) is nonnegative in higher
dimension), there is a better result: $e^{nw}$ is not only a strong $A_\infty$ weight, but also an $A_1$ weight.
\begin{theorem}\label{A1}
If $(M^4,g)$ (resp. $(M^n,g)$) is the manifold described in Theorem \ref{main} (resp. Remark
\ref{r1}), also assume Q-curvature is nonnegative (resp. P(x) is nonnegative in higher dimension),
then $e^{4w}$ (resp. $e^{nw}$) is an $A_{1}$ weight function.
\end{theorem}

Let me summarize the structure of this paper before going into any detail. In section 2 some basic properties of quasi-conformal map are reviewed, and Bonk, Heinonen and Saksman's result \cite{bonk} on quasi-conformal
map by using quasi-conformal flow is presented. In section 3 we are going
to prove several propositions and lemmas which finally lead to the proof of Theorem \ref{main}. This is the main
part of the paper. In section 4 we are going to give a much shorter and very different proof of isoperimetric inequality in the special case $Q\geq 0$, which relates to $A_p$ weights. Definitions and preliminaries will be included there. Finally in section 5, we will prove the bi-Lipschitz parametrization result by using Theorem \ref{main}.\\

\noindent \textbf{Acknowledgments:} The author would like to thank her advisor Professor Alice Chang for suggesting this interesting problem and is grateful to both Alice Chang and Paul Yang for their constant help, support and many stimulating discussions.

\section{Quasiconformal map}
In this section, we are going to present some preliminary facts about quasi-conformal map. An orientation preserving homeomorphism $f:\mathbb{R}^n\rightarrow \mathbb{R}^n $ is called
quasiconformal (for simplicity, we denoted it as q.c. map) if
$$\sup_{x\in \mathbb{R}^n} H(x,f)<\infty,$$
where
$$H(x,f)=\limsup_{r\rightarrow 0+ }\sup_{|u-x|=|v-x|=r}\frac{|f(u)-f(x)|}{|f(v)-f(x)|}.
$$
Denote $H(f)=\sup_{x\in \mathbb{R}^n}H(x,f)<\infty$ the dilatation of a q.c. map. We call $f$ is $H$-
quasiconformal if $H(f)\leq H$. Note that the inverse of a q.c. map is still q.c., and
the composition of two q.c. maps is also q.c. Their dilatation relates in the following way:
$$H(f^{-1})=H(f),
$$
and
$$H(f_1\circ f_2)\leq H(f_1)\cdot H(f_2).$$
A family $\Gamma$ of q.c. maps is called uniformly $H$-quasiconformal if $\exists \quad H\geq 1$,
 s.t. $H(f)\leq H$, $\forall f\in \Gamma$. The Jacobian of a q.c. map $J_{f}(x)=det(Df(x))$ is well
defined almost everywhere. $J_{f}(x)>0$ because $f$ is orientation preserving map.
The analytic definition of q.c. map is that $Df\in L_{loc}^{n}(\mathbb{R}^n)$ and
$C^{-1}|Df(x)|^n\leq  J_{f}(x)\leq C |Df(x)|^n$ almost everywhere. For equivalence of these two definitions
we refer to Vaisala's monograph \cite{Va}. From this analytic definition, we see
that the Jacobian of q.c. map is always locally integrable, because the q.c. map has the change of variable formula: let $f:\mathbb{R}^n\rightarrow \mathbb{R}^n$ is q.c., then
\begin{equation}
\int_{\mathbb{R}^n}u\circ f(x)J_f(x) dx=\int_{\mathbb{R}^n}u(y)dy, \quad    \mbox{$\forall$ integrable function $u$ on $\mathbb{R}^n$}.
\end{equation}
Now we come to look at a simple example
of q.c. map and its Jacobian:
\begin{example}\label{cone}
$\varphi_\beta(x)=\frac{x}{|x|^\beta}$ is a q.c. map from $\mathbb{R}^n\rightarrow \mathbb{R}^n$ when $\beta<1$. It has
dilatation $H(\varphi_\beta)=\frac{1}{1-\beta}$. And its Jacobian is
\begin{equation}
J_{\varphi_\beta}(x)=\frac{(1-\beta)}{|x|^{n\beta}}
\end{equation}
\end{example}

Note in the example, $\beta$ could be negative. The reason that $\beta<1$ is natural. In fact if $\beta=1$, $\varphi_1=\frac{x}{|x|}$ is not a
homeomorphism from $\mathbb{R}^n\rightarrow \mathbb{R}^n$. For $\beta >1$, $J_{\varphi_\beta}(x)=\frac{1-\beta}{|x|^{n\beta}}$
 is not integrable on a neighborhood of $0$, which violates the analytic definition of q.c. map.\\
The next beautiful result is by Gehring, which relates the Jacobian of q.c. map to the weight functions.
\begin{theorem}(Gehring)\cite{Gehring}The quasiconformal Jacobian is always a strong $A_\infty$ weight.
\end{theorem}We will present the definition of $A_p$ weight and its properties in section 4. As a consequence of the fact that a strong $A_\infty$ weight(actually $A_{\infty}$ is enough) has inverse Holder inequality, one has:
\begin{lemma}\cite[Chap. V]{stein}Let $f:\mathbb{R}^n\rightarrow \mathbb{R}^n$ is an $H$ quasi-conformal. Then there exist constants
$\alpha=\alpha(n,H)>0$ and $C=C(n,H)>0$ such that for every ball
$B\subset \mathbb{R}^n$,
\begin{equation}
\frac{1}{|B|}\int_{B}J_f(x)^{-\alpha}  dx\leq C(\frac{|B|}{|f(B)|})^\alpha.
\end{equation}
\end{lemma}
We will also use the weak convergence property of q.c. maps, see \cite{Va}: a family of uniformly $H$-q.c. maps has the weak compactness property when they are well normalized:
$\{f_k\} $ are uniformly $H$ q.c. and $f_{k}(0)$ is uniformly bounded and suppose $\exists$ $C\geq 1$
and a ball $B$, s.t.
\begin{equation}
\frac{1}{C}\leq \int_{B}J_{f_k}(x)dx\leq C, \forall \ k=1,2,\dots,
\end{equation}
then it has a subsequence, still denoted as $f_k$, that weakly converges to a $H'$ q.c. map where $H'$ depends on $H$ and $n$,
\begin{equation}\int_{\mathbb{R}^n} \chi_{\Omega}J_{f_k}(x)dx\rightarrow \int_{\mathbb{R}^n}\chi_{\Omega}J_{f}(x)dx \quad \mbox{as}
\quad k\rightarrow \infty.\end{equation}
After all the preliminary facts of q.c. maps, we come to present Bonk, Heinonen and Saksman's result.
\begin{theorem}(Theorem 1.2 in \cite{bonk}) \label{bonk}For each $n\geq 2$ and $H\geq1$
there exist constants $\epsilon_0=\epsilon_0(H,n)$, $H'=H'(H,n)$ and $C=C(H,n)$ with the following
property: If $g:\mathbb{R}^n\rightarrow \mathbb{R}^n$ is an $H$-q.c. map and if $w$ is a function on $\mathbb{R}^n$, taking the form
\begin{equation}w(x)=\mathfrak{L}\mu \circ g:= \int_{\mathbb{R}^n}\log\frac{1}{|g(x)-y|} d\mu(y),
\end{equation}
where $\mu$ is a signed Radon measure with finite total variation, satisfying
\begin{equation}
\int_{\mathbb{R}^n}\log^{+}|y|d|\mu|(y)<\infty,
\end{equation}
and
$$\|\mu\|=\int_{\mathbb{R}^n}d|\mu|(y)<\epsilon_0$$ then there exists a $H'$-q.c. map $f:\mathbb{R}^n\rightarrow \mathbb{R}^n$, s.t.
\begin{equation}
C^{-1} e^{nw}\leq J_{f}(x)\leq C e^{nw}.
\end{equation}
\end{theorem}
They also get when taking $g=id$ (where $H=1$) then one takes $\epsilon_0=\epsilon_0(n)=\frac{n}{128}
\cdot 12^{-2n}e^{-4(n-1)n}$. The proof of this theorem uses quasi-conformal flow, which has many
 applications in $n=2$, but rarely in higher dimensions. We are going to use this result in
 our proof in section 3.

\section{Proof of Theorem {\ref{main}}}Suppose $\mu$ is a signed Radon measure with $\int_{\mathbb{R}^n}d|\mu|(y)<\infty$ and
$\int_{\mathbb{R}^n}d\mu(y)<1$. Define the logarithmic potential
$$\dstyle \mathfrak{L}(\mu)(\cdot):=\int_{\mathbb{R}^n }\log\frac{1}{|\cdot-y|}d\mu(y), $$
for $\mu$ satisfies 
\begin{equation}\label{add}\int_{\mathbb{R}^n}\log^+|y|d|\mu|(y)<\infty.\end{equation}
Such integral is finite for a.e. $x\in \mathbb{R}^n$. This is because the maximal function $M\mu$ of $\mu$,
$$ \textsl{M}\mu(x) :=\sup_{r>0}\frac{1}{|B(x,r)|}\int_{B(x,r)}d|\mu|, x\in \mathbb{R}^n$$ is finite for almost
 every $x\in \mathbb{R}^n$, and whenever $\textsl{M}\mu(x)<\infty $,
 $$\dstyle \int_{\mathbb{R}^n} \log^+\left(\frac{1}{|x-y|}\right)d|\mu|(y)<\infty.$$
Indeed
\begin{equation}\begin{array}{lcl}
\dstyle \int_{\mathbb{R}^n}\log^+(\frac{1}{|x-y|})d|\mu|(y)&=&\dstyle \int_{0}^\infty \mu\{y|\log^+(\frac{1}{|x-y|})\geq \lambda\}
d\lambda\\
&\leq&\dstyle \int_{0}^\infty \textsl{M}\mu(x) |B(x, e^{-\lambda})|d\lambda\\
&=&\dstyle \textsl{M}\mu(x)\int_{0}^\infty \left|\{y|\log^+(\frac{1}{|x-y|})\geq \lambda\}\right|d\lambda\\
&=& \dstyle \textsl{M}\mu(x) \int_{\mathbb{R}^n} \log^+(\frac{1}{|x-y|})dx,
\end{array}
\end{equation}
where $|\cdot |$ is the Lebesgue measure of set on $\mathbb{R}^n$.
Together with (\ref{add}), one knows that $\mathfrak{L}(\mu)(x)<\infty$ if $x$ satisfies $\textsl{M}\mu(x)<\infty $.

Instead of $\mu$, we consider a special kind of measure $\nu=\beta \delta_0+\frac{1}{4\pi^2}
\mu_{\epsilon}$ first, and prove there is a q.c. map whose Jacobian is comparable to $e^{n\mathfrak{L}(\nu)}$.
\begin{proposition}\label{3.1}
For each $H>1$, if constant $\beta<1$ satisfies $\frac{1}{1-\beta}\leq H$, and $\epsilon $ satisfies $\epsilon
\leq \epsilon_{0}(H,n)$ where $\epsilon_0(H, n)$ is the function mentioned in Theorem
\ref{bonk}, then for measure $\nu=\beta \delta_0
+\mu_{\epsilon}$, where $\delta_0$ is the dirac measure at $0$
and $\mu_\epsilon $ is a signed Radon measure on $\mathbb{R}^n$ with total
variation $$\|\mu_\epsilon\|=\int_{\mathbb{R}^n}d|\mu_\epsilon|(y)<\epsilon ,$$
and \begin{equation}\label{log}\int_{\mathbb{R}^n}\log^+|y|d|\mu_\epsilon|(y)<\infty,
\end{equation}
there exists a $H'$-q.c. map $f:\mathbb{R}^n\rightarrow \mathbb{R}^n$, such that

\begin{equation}
C^{-1} e^{n\mathfrak{L}(\nu)(x)}\leq J_{f}(x)\leq C e^{n \mathfrak{L}(\nu)(x)},
\end{equation}
where $C=C(\beta,H,n)$ and $H'=H'(\beta,H,n)$.
\end{proposition}

\begin{proof}
Recall in the Example \ref{cone}, $\varphi_{\beta}=\frac{x}{|x|^\beta}$. Take $g=\varphi_{\beta}^{-1}$
which is also a q.c map with $H(g)=H(\varphi_{\beta}^{-1})=\frac{1}{1-\beta}<H$. Thus $g$ is
 $H$-q.c. map. Apply Theorem \ref{bonk} to such $H$ and $g$ and $\mu_\epsilon$. By our assumption
 $\|\mu_{\epsilon}\|\leq \epsilon_0(H,n)$, every condition in Theorem \ref{bonk} is satisfied.
Thus there exists a $H''$-q.c map $f_\epsilon:\mathbb{R}^n\rightarrow \mathbb{R}^n$, where $H''=H''(H,n)$, such that
\begin{equation}
\dstyle C^{-1} e^{n \mathfrak{L}(\mu_\epsilon)\circ g(y)}\leq J_{f_{\epsilon}}(y)\leq C e^{n \mathfrak{L}(\mu_\epsilon)\circ g(y)}.
\end{equation}
i.e.
\begin{equation}\label{2.1}
\dstyle C^{-1} e^{n \int_{\mathbb{R}^n}\log\frac{1}{|g(y)-z|}d\mu_{\epsilon}(z)}\leq J_{f_{\epsilon}}(y)\leq
 C e^{ n \int_{\mathbb{R}^n}\log\frac{1}{|g(y)-z|}d\mu_{\epsilon}(z)},
\end{equation}
where $C=C(H,n)$.
Define $f(x)= f_{\epsilon}\circ\varphi_{\beta}(x)$. Then $f:\mathbb{R}^n\rightarrow \mathbb{R}^n$ is a $H'$-q.c. map, where
$H'=\frac{1}{1-\beta}\cdot H''$, thus $H'$ depends on $\beta, H$ and $n$, and $f$ has property
\begin{equation}\label{2.4}
\begin{array}{lcl}
\dstyle J_{f}(x)=J_{f_{\epsilon}\circ\varphi_{\beta}}(x)&=&\dstyle J_{f_{\epsilon}}(\varphi_{\beta}(x))
\cdot J_{\varphi_{\beta}}(x).\\
\end{array}
\end{equation}
Plug $y=\varphi_{\beta}(x)$ into (\ref{2.1}),

\begin{equation}\label{2.2}
\dstyle C^{-1} e^{n\int_{\mathbb{R}^n}\log \frac{1}{|x-z|}d\mu_{\epsilon}(z)}\leq J_{f_{\epsilon}}
(\varphi_{\beta}(x))\leq C e^{ n \int_{\mathbb{R}^n}\log \frac{1}{|x-z|}d\mu_{\epsilon}(z)}.
\end{equation}
On the other hand,
$$\mathfrak{L}(\beta \delta_{0})(x)=\int_{\mathbb{R}^n}\log\frac{1}{|x-z|}d\beta\delta_{0}(z)=-\beta \log|x|,$$
thus as $$J_{\varphi_{\beta}}(x)= \frac{1-\beta}{|x|^{n\beta}},$$
\begin{equation}\label{2.3}
\dstyle C^{-1}(\beta) e^{n \mathfrak{L}(\beta \delta_{0})(x)}\leq J_{\varphi_{\beta}}(x)\leq C(\beta)  e^{n \mathfrak{L}(\beta \delta_{0})(x)}.
\end{equation}
Combining (\ref{2.2}) and (\ref{2.3}) and plug into (\ref{2.4}) we get
$$C^{-1} e^{n\mathfrak{L}( \nu)(x)} \leq J_{f}(x)\leq C e^{n \mathfrak{L}( \nu)(x)},$$
where $C=C(\beta,H,n)$.
Thus $f$ is the $H'$-q.c. map that we want.
\end{proof}
Now we want to show a modified version of the above proposition. The purpose of this modification is to delete condition (\ref{log}) and to fit the problem into the setting of Theorem \ref{main}. The proof follows the same spirit of Proposition 7.2 in \cite{bonk}. Before the statement of this proposition, some more terminology needs to be introduced. As mentioned before, for a signed measure of finite absolute integral, the maximal
function $\textsl{M}\nu$ of $\nu$,
$$ \textsl{M}\nu(x) :=\sup_{r>0}\frac{1}{|B(x,r)|}\int_{B(x,r)}d|\nu|, x\in \mathbb{R}^n$$
is finite for almost every point. Thus one can choose a point $x_0$ very close to $0$, such that
$\textsl{M}\nu(x_0)<\infty$. By the argument before Proposition \ref{3.1},  
\begin{equation}\label{3.3}
\int_{\mathbb{R}^n}\log^+(\frac{1}{|x_0-y|})d|\nu|(y)<\infty.
\end{equation}
Define 
\begin{equation}
\mathfrak{\tilde{L}}(\nu) := \int_{\mathbb{R}^n}\log \frac{|x_0-y|}{|x-y|}d\nu(y).
\end{equation}
Then it is finite whenever $\textsl{M}\nu(x)<\infty$, hence for almost every $x\in \mathbb R^n$ it is finite.
Now we want to modify Proposition \ref{3.1} by deleting condition (\ref{log}) while
replacing $\mathfrak{L}(\nu) $ to $\mathfrak{\tilde{L}}(\nu) $. 
\begin{proposition}\label{3.2}
For each $H>1$, if constant $\beta<1$ satisfies $\frac{1}{1-\beta}\leq H$, and $\epsilon $ satisfies $\epsilon
\leq \epsilon_{0}(H,n)$ where $\epsilon_0(H, n)$ is the function mentioned in Theorem
\ref{bonk}, then for measure $\nu=\beta \delta_0
+\mu_{\epsilon}$, where $\delta_0$ is the Dirac measure at $0$
and $\mu_\epsilon $ is a signed Radon measure on $\mathbb{R}^n$ with total
variation $$\int_{\mathbb{R}^n}d|\mu_\epsilon|(y)<\epsilon ,$$ there exists a $\tilde{H}$-q.c. map $f:\mathbb{R}^n\rightarrow \mathbb{R}^n$, such that

\begin{equation}
C^{-1} e^{n \mathfrak{\tilde{L}}(\nu)(x)}\leq J_{f}(x)\leq C e^{n \mathfrak{\tilde{L}}(\nu)(x)},
\end{equation}
where $C=C(\beta,H,n)$ and $\tilde{H}=\tilde{H}(\beta,H,n)$.
\end{proposition}

We need a lemma to prove this proposition. The lemma is due to dominated convergence theorem in this setting. It was shown in \cite{bonk}.
\begin{lemma}\label{L3.5}(Lemma 7.1 in \cite{bonk}) Assume $\|\mu\|<1$. Define $\mu_k=\mu|_{B(0,k)}$ for $k=1,2,...$. Then
$e^{n \mathfrak{\tilde{L}}(\mu)}$ is locally integrable and for every ball $B\subseteq \mathbb{R}^n$(not necessarily centered at $0$)
we have
\begin{equation}\int_{B}|e^{n\mathfrak{\tilde{L}}(\mu_k)(x)}-e^{n \mathfrak{\tilde{L}}(\mu)(x)}|\rightarrow 0, \quad \mbox{as} \quad
k\rightarrow \infty.
\end{equation}
\end{lemma}
\begin{proof}
It is easy to see $\mathfrak{\tilde{L}}(\mu_k)(x)\rightarrow \mathfrak{\tilde{L}}(\mu)(x)$ as $k\rightarrow \infty$ for almost every $x\in \mathbb{R}^n$. Define
$$U(x)=\int_{\mathbb{R}^n}\log^+\left(\frac{|x-y|}{|x_0-y|}\right)d\mu_{-}(y)+\int_{\mathbb{R}^n}\log^-
\left(\frac{|x-y|}{|x_0-y|}\right)d\mu_{+}(y)$$
Then
$$e^{\mathfrak{\tilde{L}}(\mu_k)(x)}\leq e^{U(x)}$$
for every $k\in \mathbb{N}$ and every $x\in \mathbb{R}^n$. The lemma follows from the dominated convergence theorem
as long as we can show that
\begin{equation}\label{3.4}\int_{B}e^{nU(x)}dx<\infty,\end{equation}
for every ball $B\subseteq \mathbb{R}^n$. Take $R$, such that $B\subseteq B(0,R)$, then for all $x\in B$, we have
$$\log^+\left( \frac{|x-y|}{|x_0-y|}\right)\leq 1+\log^{+}(R+|x_0|)+\log^+\left(\frac{1}{|x_0-y|}\right),$$
where we've used $\log^+(1+xy)\leq 1+\log^+ x+\log^+ y$.
Together with (\ref{3.3}) we get a uniform estimate
$$\int_{\mathbb{R}^n}\log^+\left(\frac{|x-y|}{|x_0-y|}\right)d\mu_{-}(y)\leq C$$
for $x\in B$.
If $\mu_+=0$ then (\ref{3.4}) is automatically true. Otherwise, define a probability measure $\nu=\frac{1}{\|\mu_+\|}\mu_+$,
where $\|\mu_+\|=\int_{\mathbb{R}^n}d|\mu_+|(y)$. Let $\lambda=n\|\mu_+\|<n$, we have
\begin{equation*}
\begin{array}{lcl}
\dstyle \int_{B}e^{nU(x)}dx&\leq& C\dstyle \int_{B}exp\left(n\int_{\mathbb{R}^n}\log^+(\frac{1}{|x-y|})d\mu_+(y)\right)dx\\
&=&\dstyle \int_{B}exp\left (\lambda\int_{\mathbb{R}^n}\log^+(\frac{1}{|x-y|})d\nu(y)\right)dx\\
&\leq &\dstyle \int_{B}\int_{\mathbb{R}^n} \max\{1,\frac{1}{|x-y|^\lambda}\}d\nu(y)dx<\infty.
\end{array}
\end{equation*}
This finishes the proof of the lemma.
\end{proof}

\begin{proof} of Proposition \ref{3.2}: By the choice of $x_0$ before the statement of Proposition \ref{3.2} and condition (\ref{log}), $\mathfrak{L}(\nu)(x)$ and $\mathfrak{\tilde{L}}(\nu)(x)$ are both finite whenever $\textsl{M}\nu(x)<\infty$. Hence for almost every $x\in \mathbb{R}^n$ they are both finite. Suppose $\nu$ is compactly supported, then $\mathfrak{L}(\nu)$ and $\mathfrak{\tilde{L}}(\nu)$ differ by a constant when they are both finite. Therefore there exists a constant $c>0$ such that
$e^{n\mathfrak{L}(\nu)(x)}=ce^{n\mathfrak{\tilde{L}}(\nu)(x)}$ a.e. $x$. Here $c$ depends on the size of $\textit{supp}(\nu)$. Apply the Proposition \ref{3.1}
to $\nu$ and multiply the q.c. map obtained in this way by a constant(in each direction), we get an $H'$-q.c. map $f:\mathbb{R}^n\rightarrow
\mathbb{R}^n$ such that
\begin{equation}
C^{-1} e^{n \mathfrak{\tilde{L}}(\nu)(x)}\leq J_{f}(x)\leq C e^{n\mathfrak{\tilde{L}}(\nu)(x)},
\end{equation}
for almost every $x\in \mathbb{R}^n$. Notice both $H'=H'(\beta,H,n)$, $C=C(\beta,H,n)$ are independent of the
size of $\textit{supp}(\nu)$.\\

For an arbitrary $\nu$ satisfying the hypotheses of the Proposition, we run an approximation argument.
By the argument before Proposition \ref{3.2} we know
\begin{equation}\label{3.7}\int_{\mathbb{R}^n}\log^+(\frac{1}{|x_0-y|})d|\mu|(y)<\infty,
\end{equation}
thus $\mathfrak{\tilde{L}}(\nu)$ is finite whenever $\mathfrak{L}(\nu)$ is finite. And they are both finite for almost every $x\in \mathbb{R}^n$.
Define $\nu_k=\nu|_{B(0,k)}=\beta \delta_{0}+\mu_\epsilon|_{B(0,k)}$ for $k=1,2,...$ similar to that in Lemma \ref{L3.5}. Thus for each $k$, $\nu_k$ is compactly
supported with $\|\mu_\epsilon|_{B(0,k)}\|<\epsilon$,
$\textsl{M}\nu_k(x_0)\leq \textsl{M}\nu(x_0)<\infty$. Apply the first part of the proof to $\nu_k$, we can find
$c_k$ for each $k$, such that $e^{n\mathfrak{L}(\nu_k)}=c_k e^{n\mathfrak{\tilde{L}}(\nu_k)}$. This $c_k$ 
depend on the size of compact support, thus on $k$. To eliminate the effect of $c_k$,
 multiply each q.c. map obtained for $\nu_k$ by constant $\sqrt[n]{c_k}$(in each direction), the dilatation $H'$ of these q.c.
maps doesn't change. Thus there exists a sequence of $H'$-q.c. map
$f_k:\mathbb{R}^n\rightarrow \mathbb{R}^n$,
$H'=H'(\beta,H,n)$ doesn't depend on $k$, such that
\begin{equation}\label{3.6}
C^{-1} e^{n \mathfrak{\tilde{L}}(\nu_k)(x)}\leq J_{f_k}(x)\leq C e^{n \mathfrak{\tilde{L}}(\nu_k)(x)},
\end{equation}
where $C=C(\beta,H,n)$.
By Lemma \ref{L3.5},
\begin{equation}\label{3.8}\int_{B}|e^{n \mathfrak{\tilde{L}}(\nu_k)(x)}-e^{n\mathfrak{\tilde{L}}(\nu)(x)}
|dx \rightarrow 0 \quad \mbox{as} \quad k  \rightarrow \infty ,\end{equation}
for every ball $B\subseteq \mathbb{R}^n$. Throughout the proof of Lemma \ref{L3.5}, it also shows
$$\dstyle \int_{B}e^{n\mathfrak{\tilde{L}}(\nu)(x)}dx<\infty.$$
\noindent This together with (\ref{3.6}) implies that there exists a constant $C\geq 1$ (depending on the sequence of q.c. maps and the ball $B$, i.e. $C=C(f_k, B)$), such that
$$\frac{1}{C}\leq \int_{B}J_{f_k}(x)
\leq C,$$ for all $k\in \mathbb{N}$. Without loss of generality, we can make $f_k(0)=0$. By the weak compactness of q.c. map discussed in Section 2(see also \cite{Va} and\cite{Rickman}(Lemma 8.8)), we know
there exists a subsequence of $H'$-q.c. maps, still denoted as $f_k$, such that $f_k$ converges weakly to an
$\tilde{H}$-q.c. map $f:\mathbb{R}^n\rightarrow \mathbb{R}^n$ where $\tilde{H}$ still depends on $\beta,H,n$. Now the
weak convergence property of Jacobian says
\begin{equation}\int_{\mathbb{R}^n} \chi_{\Omega}J_{f_k}(x)dx\rightarrow \int_{\mathbb{R}^n}\chi_{\Omega}J_{f}(x)dx \quad \mbox{as}
\quad k\rightarrow \infty,\end{equation}
for any $\Omega\subseteq \mathbb{R}^n$.
Together with (\ref{3.6}), (\ref{3.8}), this implies
\begin{equation}
\dstyle C^{-1}\int_{\Omega}e^{n \mathfrak{\tilde{L}}(\nu)(x)} dx
\leq \dstyle \int_{\Omega} J_{f}(x)dx \leq C\int_{\Omega}e^{n  \mathfrak{\tilde{L}}(\nu)(x)}dx,\\
\end{equation}
for arbitrary bounded subset $\Omega$ of $\mathbb{R}^n$. Therefore,
\begin{equation}
\dstyle C^{-1} e^{n  \mathfrak{\tilde{L}}(\nu)(x)}\leq J_{f}(x)\leq C e^{n  \mathfrak{\tilde{L}}(\nu)(x)},
\end{equation}
for almost all $x\in \mathbb{R}^n$, where $C=C(\beta,H,n)$. This is the desired property.
\end{proof}
From now on, we are trying to smooth the measure $\beta\delta_0$ while preserving the good property that it associates with a q.c. map of certain dilatation. The model case $\beta\delta_0$ gives very good intuition that we are starting with the cone
map $\varphi_\beta=\frac{x}{|x|^\beta}$, and then composite it with another q.c map which only
contributes to a measure $\mu_\epsilon$ of very small variation. However for later use, we need a measure with some smoothness, by which I mean
it should be written as $f(x)dx$ s.t. $f$ is continuous. So we want to smooth
$\beta\delta_0$
a little bit, but still not far away from the cone case: basically we want to find a smooth totally integrable
measure of the form $m=f(x)dx$ which is compactly supported in $B(0,1)$, $\dstyle \int d m=\beta$ and $e^{n  \mathfrak{L}(m)(x)}$ is
comparable to the Jacobian of some q.c map $\psi_\beta$, with $H(\psi_\beta)$ close to $\dstyle H(\varphi_
\beta)=\frac{1}{1-\beta}$. It is not hard. We now explicitly construct a smooth
q.c. map together with the measure $m_\beta$. Define map $\psi_{\beta}$ by polar coordinates: $\psi_{\beta}$ maps $r\mapsto R(r) $, $\theta\mapsto \theta$, where
$R(r)$ is the smooth function satisfying
\begin{equation}
\left\{\begin{array}{ll}
R(r)&=r, \quad \mbox{if}\quad r\in [0,\delta];\\
R(r)\ & \mbox{is smooth with}\quad 1-\beta \leq R'(r
)\leq 1\quad\\
 &\mbox{and satisfies q.c. condition } (\ref{q.c.}), \quad \mbox{if} \quad r\in(\delta,1-\delta); \\
R(r)&=r^{1-\beta},\quad \mbox{if}\quad r\in [1-\delta,\infty),\\ \end{array}\right.
\end{equation}
where q.c. condition means:
\begin{equation}\label{q.c.}
1-\beta \leq \frac{r\cdot R'(r)}{R}\leq 1.
\end{equation}
The (\ref{q.c.}) condition is easy to achieved because $\frac{r\cdot R'(r)}{R}=1 $ near $0$, and
$\frac{r\cdot R'(r)}{R}=1-\beta$ for $r\geq 1-\delta$.
\noindent $R$ is a smooth function on $[0,\infty)$, thus $\psi_\beta\in C^{\infty}(\mathbb{R}^n; \mathbb{R}^n)$.
Also $R'(r)>0$ for $r\in[0,\infty)$, thus $\psi_\beta: \mathbb{R}^n\rightarrow \mathbb{R}^n$ is a homeomorphism. Moreover, condition
(\ref{q.c.}) on $B(0,1)$ implies
the upper bound of dilatation $H(\psi_\beta)= \frac{1}{1-\beta}$.
In fact, the induced metric in polar coordinate by $\psi_\beta$ is:
$$dR^2+R^2 d\theta^2=R'^2(r)dr^2+R^2d\theta^2. $$
When (\ref{q.c.}) is satisfied,
$$R'^2(r)(dr^2+r^2d\theta^2) \leq dR^2+R^2 d\theta^2 \leq R'^2(r)(dr^2 + \frac{1}{(1-\beta)^2} r^2d\theta^2),$$
which means $\psi_\beta$ is a $\frac{1}{1-\beta}$-q.c. map.\\
$\log\frac{1}{|x|}$ is the fundamental solution of $\Delta^{n/2}$ on $\mathbb{R}^n$,
i.e. $$\Delta^{n/2}\log\frac{1}{|x|}=c_n\delta_0(x),$$ where $c_{n}$ is a constant depending
on $n$, for example $c_2=2\pi$,$c_4=8\pi^2$.
Define the signed Radon measure:
\begin{equation}
\label{m beta}m_\beta:=\frac{1}{c_n}\Delta^{n/2}\left(\frac{1}{n}\log J_{\psi_{\beta}}\right)(x)dx.\end{equation}

i) $m_\beta$ is compactly supported on
 $B(0,1)$, since $\psi_\beta$ coincides with $\varphi_\beta$ outside the ball $B(0,1-\delta)$. Also $m_\beta$ is a smooth measure. This is because $1-\beta \leq R'(r)\leq 1$ and $D\theta=id_{n-1}$
thus $J_{\psi_\beta}(x)>0$ for any $x\in \mathbb{R}^n$. Thus $\log J_{\psi_{\beta}}(x)$ is a smooth function defined
everywhere.\\

ii) $J_{\psi_{\beta}}\simeq e^{n  \mathfrak{L}(m_\beta )(x)}$.
\begin{proof}
\begin{equation}\label{3.9}
\begin{array}{lcl}
\mathfrak{L}(m_\beta)(x)&=&\dstyle \frac{1}{c_n}\int_{\mathbb{R}^n}\log \frac{1}{|x-y|}\cdot \Delta^{n/2}(\frac{1}{n}\log J_{\psi_{\beta}})(y)dy,\\
 &=&\dstyle \frac{1}{c_n}\int_{\mathbb{R}^n}\Delta^{n/2}(\log\frac{1}{|x-y|})\cdot \frac{1}{n}\log J_{\psi_\beta}(y)dy+C,\\
 &=&\dstyle \int_{{\mathbb{R}^n}}\delta_{x}(y) \frac{1}{n}\log J_{\psi_\beta}(y)dy+C,\\
 &=& \dstyle \frac{1}{n}\log J_{\psi_\beta}(x)+C,
\end{array}
\end{equation}
where the second equality is because the derivatives of $\log|x-y|$ has good decay at $\infty$, and the constant $C$ comes from the boundary term at $\infty$ when taking integration by part on $\mathbb R^n$. 
(\ref{3.9}) is the same as
$$J_{\psi_{\beta}}(x)\simeq e^{n\mathfrak{L}(m_\beta) (x)},$$ thus we proved ii).
\end{proof}
iii) $\int_{\mathbb{R}^n}dm_\beta=\beta$.
\begin{proof}
Notice, in polar coordinates $\varphi_{\beta}(x)$ is mapping $r\mapsto r^{1-\beta}$, $\theta\mapsto\theta$,
thus $$\psi_{\beta}=\varphi_{\beta}(x) \quad \mbox{when}\quad x\in \mathbb{R}^n\setminus B(0,1-\delta).$$
Therefore,
$$ \Delta^{n/2}(\frac{1}{n}\log J_{\psi_\beta})(x)=\Delta^{n/2}(\frac{1}{n}\log J_{\varphi_\beta})(x)=0,
$$
for $x \in \mathbb{R}^n \setminus B(0,1)$.
\begin{equation}\begin{array}{lcl}
\dstyle \int_{\mathbb{R}^n}dm_\beta&=&\dstyle \frac{1}{c_n}\int_{\mathbb{R}^n} \Delta^{n/2}(\frac{1}{n}\log J_{\psi_\beta})(x)dx\\
&=&\dstyle  \frac{1}{c_n} \int_{B(0,1)} \Delta^{n/2}(\frac{1}{n}\log J_{\psi_\beta})(x)dx.\\
\end{array}
\end{equation}
We use divergence theorem and the fact that $\psi_\beta$ coincides with $\varphi_\beta$ on a
neighborhood of $\partial B(0,1)$ to get,
\begin{equation}
\begin{array}{lcl}
&&\dstyle \frac{1}{c_n} \int_{B(0,1)} \Delta^{n/2}(\frac{1}{n}\log J_{\psi_\beta})(x)dx\\
&=&\dstyle \frac{1}{c_n}\int_{\partial B(0,1)}\nabla \Delta^{(n/2-1)}(\frac{1}{n}\log J_{\psi_\beta})\cdot d\vec n\\
&=&\dstyle \frac{1}{c_n}\int_{\partial B(0,1)}\nabla \Delta^{(n/2-1)}(\frac{1}{n}\log J_{\varphi_\beta})\cdot d\vec n\\
&=&\dstyle \frac{1}{c_n}\int_{B(0,1)}\Delta^{n/2}(\beta \log \frac{1}{|x|})dx\\
&=&\dstyle \beta.
\end{array}
\end{equation}
Thus we finished the proof of iii).
\end{proof}
With property i),ii),iii) we get another model case q.c map $\psi_\beta$ such that its corresponding measure
$m_\beta$ is a smooth measure. Now we simply repeat the same argument as in Proposition \ref{3.1} and Proposition 3.2 with $\beta \delta_0$ replaced by the smooth measure $m_\beta$, and map $\varphi_\beta$ replaced by $\psi_\beta$. 

\begin{proposition}\label{3.1'}
For each $H>1$, if constant $\beta<1$ satisfies $\frac{1}{1-\beta}\leq H$, and $\epsilon $ satisfies $\epsilon
\leq \epsilon_{0}(H,n)$ where $\epsilon_0(H, n)$ is the function mentioned in Theorem
\ref{bonk}, then for measure $\nu=m_\beta
+\mu_{\epsilon}$, where $\mu_\epsilon $ is a signed Radon measure on $\mathbb{R}^n$ with total
absolute integral $$\|\mu_\epsilon\|=\int_{\mathbb{R}^n}d|\mu_\epsilon|(y)<\epsilon ,$$
and \begin{equation}\label{log'}\int_{\mathbb{R}^n}\log^+|y|d|\mu_\epsilon|(y)<\infty,
\end{equation}
there exists a $H'$-q.c. map $f:\mathbb{R}^n\rightarrow \mathbb{R}^n$, such that

\begin{equation}
C^{-1} e^{n\mathfrak{L}(\nu)(x)}\leq J_{f}(x)\leq C e^{n \mathfrak{L}(\nu)(x)},
\end{equation}
where $C=C(\beta,H,n)$ and $H'=H'(\beta,H,n)$.
\end{proposition}

Similar to Proposition \ref{3.2}, we now want to show a modified version of Proposition \ref{3.1'} for $\mathfrak{\tilde{L}}(\nu)$ without (\ref{log'}).
Still we choose a point $x_0$ very close to $0$, such that the maximal
function
$$ \textsl{M}\nu (x_0):=\sup_{r>0}\frac{1}{|B(x_0,r)|}\int_{B(x_0,r)}d|\nu|<\infty.$$
This is possible because the maximal function is finite almost every where. Then as the argument before,
\begin{equation}\label{3.3'}
\int_{\mathbb{R}^n}\log^+(\frac{1}{|x_0-y|})d|\nu|(y)<\infty.
\end{equation}
Thus as long as $\textsl{M}\nu(x)<\infty$,
\begin{equation}
\mathfrak{\tilde{L}}(\nu) := \int_{\mathbb{R}^n}\log \frac{|x_0-y|}{|x-y|}d\nu(y)
\end{equation}
is finite. Hence for almost every $x\in \mathbb{R}^n$ it is finite.
\begin{proposition}\label{3.2'}
For each $H>1$, if constant $\beta<1$ satisfies $\frac{1}{1-\beta}\leq H$, and $\epsilon $ satisfies $\epsilon
\leq \epsilon_{0}(H,n)$ where $\epsilon_0(H, n)$ is the function mentioned in Theorem
\ref{bonk}, then for measure $\nu=m_\beta
+\mu_{\epsilon}$, where $\mu_\epsilon $ is a signed Radon measure on $\mathbb{R}^n$ with total
absolute integral $$\int_{\mathbb{R}^n}d|\mu_\epsilon|(y)<\epsilon .$$ Then
there exists a $\tilde{H}$-q.c. map $f:\mathbb{R}^n\rightarrow \mathbb{R}^n$, such that

\begin{equation}
C^{-1} e^{n \mathfrak{\tilde{L}}(\nu)(x)}\leq J_{f}(x)\leq C e^{n \mathfrak{\tilde{L}}(\nu)(x)},
\end{equation}
where $C=C(\beta,H,n)$ and $\tilde{H}=\tilde{H}(\beta,H,n)$.
\end{proposition}
\noindent To repeat the same argument as in Proposition \ref{3.1} amd \ref{3.2} with $\beta\delta_0$ by $m_\beta$, one can get the proof of Proposition \ref{3.1'} and \ref{3.2'}. Therefore we omit the proof of these two theorems here.

\begin{remark} Notice the measure $\nu$ could have total mass very close to 1. This is how our proposition is different from Theorem 1.1 in \cite{bonk}.
\end{remark}
\noindent With the above preparation, it is ready to prove Theorem \ref{main}.
\begin{proof} of theorem \ref{main}: Consider the noncompact, complete manifold $(M, e^{2w}dx^2)$ with normal metric,
$$\dstyle w(x)=\frac{1}{4\pi^2}\int_{\mathbb{R}^4}\log\frac{|y|}{|x-y|}Q(y)e^{4w(y)}dy+C.$$
Define the measure $$\mu:=\frac{1}{4\pi^2} Q(x)e^{4w(x)}dx,$$
then
$$\alpha:=\int_{\mathbb{R}^4}d\mu(x)<1.$$
Since $\mu$ is a smooth measure, every point $x\in \mathbb{R}^4$ has property
$$\textsl{M}(\mu)(x):= \sup_{r>0}\frac{1}{|B(x,r)|}\int_{B(x,r)}d|\mu|<\infty.$$
Thus we can just let $x_0=0$. To use Proposition \ref{3.2'} we need to choose proper $H$ and $\nu$.
We can choose any $H>\frac{1}{1-\alpha}$, for example, define $H=\frac{100}{1-\alpha}>1$. For such $H$ and dimension 4, there exists $\epsilon_0=\epsilon_0(H,4)$ as
shown in Theorem \ref{bonk}. Now, since
\begin{equation}
\dstyle \int_{M^4} |Q| dv_{M}=\dstyle \int_{\mathbb{R}^4} |Q| e^{4w}dx< \infty,
\end{equation}
there exists $R\gg 1$, such that
\begin{equation}
\dstyle \int_{\mathbb{R}^4\setminus B(0,R)} |Q| e^{4w}dx< \min\{\epsilon_0,\frac{1-\alpha}{100}\}.
\end{equation}
With this $R$, define
$$\mu_\epsilon=\mu|_{\mathbb{R}^4\setminus B(0,R)}.$$
It is obvious that $$ \int_{\mathbb{R}^4\setminus B(0,R)} d|\mu_\epsilon|(y)<\epsilon_0$$
Set $\beta=\frac{1}{4\pi^2}\int_{ B(0,R)} Q e^{4w}dx$, so $\beta<\alpha+\frac{1-\alpha}{100}$. Therefore
$\frac{1}{1-\beta}<H$. Consider $\nu=m_{\beta}+\mu_\epsilon$, where $m_\beta$ is defined by (\ref{m beta}), with the choice $x_0=0$, $\mathfrak{\tilde{L}}(\nu)$ is finite almost every where. By Proposition \ref{3.2'}
there exists a $\tilde{H}$-q.c. map $f:\mathbb{R}^4\rightarrow \mathbb{R}^4$ such that,
\begin{equation}\label{main1} C^{-1}e^{4\mathfrak{\tilde{L}}(\nu)(x)}\leq J_{f}(x)\leq C e^{4\mathfrak{\tilde{L}}(\nu)(x)},\end{equation} where
$C=C(\beta,H,4)$, $\tilde{H}=\tilde{H}(\beta,H,4)$. Since $\beta$ and $H$ depend only on $\alpha$,
$C=C(\alpha,4)$ and $\tilde{H}=\tilde{H}(\alpha,4)$.
On the other hand, with $\nu=m_{\beta}+\mu_\epsilon$ we have, 
\begin{equation}\begin{array}{lcl}
\displaystyle\mu&=\displaystyle& \mu_\epsilon+\mu|_{B(0,R)}\\
&=&\mu_\epsilon +m_{\beta}+\mu|_{B(0,R)}-m_{\beta}\\
&=&\displaystyle \nu+(\mu|_{B(0,R)}-m_\beta)=\nu+h(x)dx,\\
\end{array}
\end{equation}
where $$h(x)=\frac{1}{8\pi^2}
\chi_{B(0,R)}(x)2Q(x)e^{4w(x)}-\frac{1}{8\pi^2}\Delta^{2}(\frac{1}{4}\log J_{\psi_\beta})(x).$$
$h(x)$ is a smooth function of compact support, so
$$\dstyle \sup_{x\in \mathbb{R}^4}h(x)\leq C(M).$$
Here $C(M)$ is a constant depending on $Q(x)e^{4w(x)}$ in $B(0,R)$, thus depending on the manifold $M$.
Meanwhile, by property ii) of $m_\beta$,
$$\int_{\mathbb{R}^4} h(x)dx=0. $$
Therefore $h(x)$ is Hardy $H^1$ function, whose $H^1$ norm is determined by $M$(more precisely, the supremum of $Q(x)e^{4w(x)}$ in $B(0,R)$ together with $R$). It is well known
that $\log$ is a BMO function and
$$\|f*g\|_{L^{\infty}}\leq \|f\|_{BMO}\cdot \|g\|_{H^1}.$$
\\
\noindent Thus,
\begin{equation}
\begin{array}{lcl}
\dstyle \left|\mathfrak{\tilde{L}}(\mu-\nu)(x)\right|&=&\dstyle \left|\int_{\mathbb{R}^4}\log
\frac{|y|}{|x-y|} h(y)dy\right| \\
&=&\dstyle \left|C_1+ \int_{\mathbb{R}^4} \log \frac{1}{|x-y|}h(y)dy\right| \\
&\leq & \dstyle |C_1|+ \|\log \|_{BMO}\cdot \|h\|_{H^1}\leq C_2.
\end{array}
\end{equation}
Here $C_1$ arises from the boundedness of both $h$ and its support $R$, thus on manifold $M^4$. Hence $C_2$ depends on the the manifold $M^4$. 
\noindent Now it follows that
\begin{equation}\label{main2}e^{-C_2}\leq e^{n\mathfrak{\tilde{L}}(\mu-\nu)(x)}\leq e^{C_2}.\end{equation}
Combining (\ref{main1}) and (\ref{main2}), we have
 $$C^{-1}e^{4\mathfrak{\tilde{L}}(\mu)(x)}\leq  J_{f} (x)\leq C e^{4\mathfrak{\tilde{L}}(\mu)(x)},$$
where $f$ is $\tilde{H}=\tilde{H}(\alpha,4)$-q.c. map and $C=C(M)$.
This finishes the proof of the main theorem.
\end{proof}

Theorem \ref{c1.2} follows directly from Theorem \ref{main}. The isoperimetric constant depends on $M^4$
since $C_2$ in the proof of the Theorem \ref{main} depends on $M^4$. But if the domain $\Omega$ is outside a very large compact set,
say $B(0,10R)$, as it is far away from $m_\beta$ and $\mu|_{B(0,R)}$, the isoperimetric constant depends only on the constant in 
$$C^{-1}e^{4\mathfrak{\tilde{L}}(\nu)(x)}\leq J_{f}(x)\leq C e^{4\mathfrak{\tilde{L}}(\nu)(x)}, $$
which depends only on $\alpha$ and $n$(here it is 4).
One can see the proofs of Theorem \ref{main} and Corollary \ref{c1.2} are based on Proposition \ref{3.2'}, which is true for general even dimensions. Thus Theorem \ref{main} is not restricted to the case $n=4$.
\begin{remark}\label{Remark 3.1} 
Moreover there is no way
to get a uniform bound independent of $M^4$ because the behavior of the manifold within the
compact set cannot be controlled. For example, consider a sequence of manifolds $\{M_{k}\}_{k=1}^{\infty}$ in Figure 3.8 with cylinder part $k\rightarrow \infty$. $M_k$ satisfies the conditions in Theorem \ref{main} with some uniform $\alpha$. On each manifold, the isoperimetric inequality holds but with constant $C_k\rightarrow \infty$ as $k\rightarrow \infty$.\\
\begin{figure}[ht]
\centering
\includegraphics[totalheight=0.15\textheight]{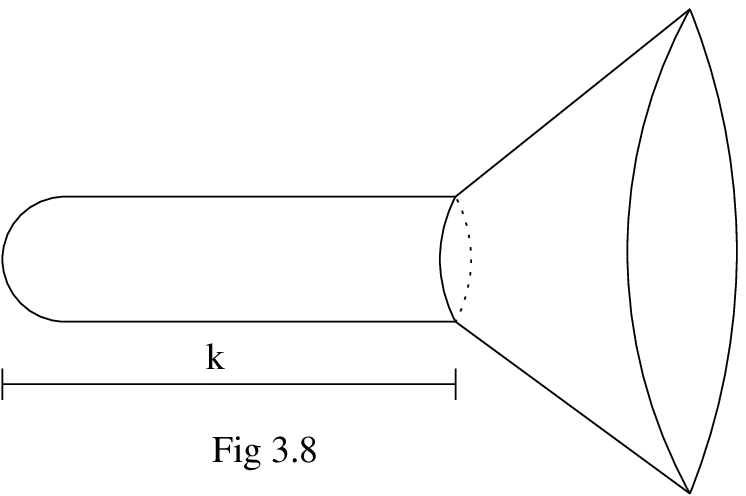}
\end{figure} 
\end{remark}

\section{Strong $A_\infty$ weight and quasiconformal map}

We begin by talking about some basics of $A_p$ and strong $A_\infty$ weights.
For a positive locally integrable function $\omega$, we call it an $A_p$ weight $p>1$, if
\begin{equation}\label{Ap}
\frac{1}{|B|}\int_{B}\omega(x)dx\cdot \left(\frac{1}{|B|}\int_{B}\omega(x)^{-{p'}/{p}}dx\right)^{{p}/{p'}}\leq A<\infty,
\end{equation}
for all balls $B$ in $\mathbb{R}^n$. Here $p'$ is conjugate to $p$: $\frac{1}{p'}+\frac{1}{p}=1$. The constant $A$ is uniform for all $B$ and we call the smallest such constant $A$ the $A_p$ bound of $\omega$.
The definition of $A_1$ weight is given by taking limit of $p\rightarrow 1$ in (\ref{Ap}),
which gives
$$\dstyle \frac{1}{|B|} \int_{B}\omega dx \leq A \omega, $$
for almost all $x\in \mathbb{R}^n$.
Thus it is equivalent to say the maximal function of the weight is bounded by the weight itself:
$$\textsl{M}\omega(x)\leq A' \omega(x).$$
The most fundamental property of $A_p$ weight is the reverse H\"{o}lder inequality:
if $\omega$ is $A_p$ weight for some $p\geq 1$, then there exists an $r>1$ and a $c>0$, such that
\begin{equation}\label{reverse holder}
\dstyle \left( \dstyle \frac{1}{|B|} \int_{B}\omega^r dx \right)^{1/r}\leq \frac{c}{|B|} \dstyle \int_{B}\omega dx,
\end{equation}
for all balls $B$.

The notion of strong $A_\infty$ weight was first introduced by David and Semme in \cite{D-S}.
A positive locally integrable function $\omega(x)$ on $\mathbb{R}^n$ is called an $A_\infty$ weight if for each $\epsilon>0$,
there is a $\delta>0$ such that if $Q\subseteq \mathbb{R}^n$ is any cube and $E\subseteq Q$ satisfies $|E|\leq \delta |Q|$, then
$\omega(E)\leq \epsilon \omega(Q)$. Here $|E|$ is the Lebesgue measure of set $E$, and $\omega(E)=\int_{E}\omega(x)dx$.
In this definition $\delta$ doesn't depend on $Q$ and $E$.
Given a weight $\omega$, and a rectifiable arc $\gamma$, the $\omega$-length of $\gamma$ is defined to be
$$\int_\gamma \omega^{\frac{1}{n}}(s)|ds|,$$
where $|ds|$ is the differential of arc length.
Now define the geodesic distance associated to $\omega$
\begin{equation}
d_{\omega}(x,y)=\inf_{\gamma}\omega\mbox{-length of }\gamma,
\end{equation}
with infimum taken over all rectifiable $\gamma$ connecting $x$ and $y$.

\noindent Define the measure distance $\delta(x,y)$ to be:
\begin{equation}
\delta(x,y)=\left(\int_{B_{x,y}}\omega(z)dz\right)^{1/n},
\end{equation}
where $B_{x,y}$ is the ball with diameter $|x-y|$ that contains $x$ and $y$. One can prove
$\delta(x,y)$ is a quasidistance, and if $\omega$ is an $A_\infty$ weight as defined before, then
\begin{equation}\label{A}
d_\omega(x,y)\leq C\delta(x,y)
\end{equation}
for all $x,y\in \mathbb{R}^n$.
If in addition to the above inequality, $\omega$ also satisfies the reverse inequality, i.e.
\begin{equation}\label{strong A}
\delta(x,y)\leq Cd(x,y),
\end{equation}
for all $x,y\in \mathbb{R}^n$, then we say $\omega$ is a strong $A_\infty$ weight.
(\ref{strong A}) is equivalent to requiring that
\begin{equation}\label{strong A'}
\omega(B)^{1/n}\leq C (\omega\mbox{-length of }\gamma).
\end{equation}
Gehring proved that the Jacobian of a quasiconformal map $f$ on $\mathbb{R}^n$ is always a strong $A_\infty$ weight. Thus by
Theorem \ref{main} $e^{nw}$ is also a strong $A_\infty$ weight, which is Corollary \ref{cA}.
Briefly, Gehring's proof goes by showing that on one hand
$$ \omega(B_{x,y})\simeq |f(x)-f(y)|^{n},
$$
and on the other hand, if $\gamma$ is any curve in the ball $B_{x,y}$ connecting $x$ and $y$, then by similar argument as in (\ref{5.1}) below,
$$|f(x)-f(y)|\leq C (\omega \mbox{-length of }\gamma).$$ Therefore (\ref{strong A'}) holds. Also it is not hard to
prove $J_{f}$ satisfies reverse H\"{o}lder inequality, so (\ref{A}) is valid.
But the converse statement is not true: not every strong $A_\infty$ weight is comparable to the Jacobian
of a q.c map. In fact Laakso \cite{La} showed that there is a strong $A_\infty$ weighted metric not bi-Lipschitz to Euclidean $\mathbb{R}^n$.
Another important property of strong $A_\infty$ weight which is proved by David and Semme in \cite{D-S} is that a strong $A_\infty$ weight has Sobolev inequality:
\begin{equation}\dstyle \left( \int_{\mathbb{R}^n}|f(x)|^{p^*}\omega(x)dx\right)^{1/{p^*}}\leq C \dstyle \left(\int_{\mathbb{R}^n}
(\omega^{-\frac{1}{n}}(x)|\nabla f(x)|)^p\omega(x)dx \right)^{1/p},
\end{equation}
where $1\leq p<n$, $p^*=\frac{np}{n-p}$. Take $p=1$, it is the usual isoperimetric inequality.
The relation between strong $A_\infty$ weight and other $A_p$ weights is briefly stated below. Every $A_1$ weight is
strong $A_\infty$, but for $p>1$ there is an $A_p$ weight which is not strong $A_\infty$. Conversely, for
any $p$ there is a strong $A_\infty$ weight but it is not $A_p$. It is easy to verify by definition the function $|x|^{\alpha}$ is $A_1$ thus strong $A_\infty$ if $-n<\alpha\leq 0$; it is not $A_1$ but still strong $A_\infty$ if $\alpha>0$.
And $|x_1|^\alpha$ is not strong $A_\infty$ for any $\alpha>0$ as one can by choosing a $\gamma$ contained in the
$x_2$-axis. As is stated in Theorem \ref{A1} we want to show when $Q$-curvature is nonnegative, $e^{nw}$ is $A_1$ weight.
\begin{proof} of Theorem \ref{A1}: To prove $e^{nw}$ is an $A_1$ weight, we need to show
\begin{equation}\label{eqnA1}
\textsl{M}(e^{nw})(x) \leq C e^{nw(x)}
\end{equation}
for all $x\in \mathbb{R}^n$.
\begin{equation}
\begin{array}{lcl}\label{73}
\dstyle \frac{\textsl{M}(e^{nw})(x)}{e^{nw(x)} }&=& \dstyle \sup_{r>0}\frac{\frac{1}{|B(x,r)|}
\dstyle \int_{B(x,r)} \exp\left( \dstyle \frac{n}{c_n}\int_{\mathbb{R}^n}\log\frac{|z|}{|y-z|} Q(z)e^{nw(z)}dz \right)dy }{ \exp\left(\dstyle \frac{n}{c_n}\int_{\mathbb{R}^n}\log\frac{|z|}{|x-z|} Q(z)e^{nw(z)}dz\right)}\\
&=&\dstyle \sup_{r>0}\frac{1}{|B(x,r)|}
\dstyle\int_{B(x,r)}\exp\left(\dstyle \frac{n}{c_n}\int_{\mathbb{R}^n}\log\frac{|x-z|}{|y-z|} Q(z)e^{nw(z)}dz \right)dy\\
\end{array}
\end{equation}
If $Q\equiv 0$, then of course (\ref{eqnA1}) is true. Otherwise, define the probability measure $\nu(z)=\frac{Q(z)e^{nw(z)}dz}{c_n\alpha}$, by the convexity of $\exp$ and nonnegativity of $Q$, we get for any $r$,
\begin{equation}\label{74}
\begin{array}{lcl}
&&\dstyle\frac{1}{|B(x,r)|}
\int_{B(x,r)}\exp\left(\dstyle \frac{n}{c_n}\int_{\mathbb{R}^n}\log\frac{|x-z|}{|y-z|} Q(z)e^{nw(z)}dz \right)dy\\
&\leq & \dstyle\frac{1}{|B(x,r)|}
\int_{B(x,r)}\dstyle \int_{\mathbb{R}^n}\left(\frac{|x-z|}{|y-z|}\right)^{n\alpha} d\nu(z)\ dy\\
&=&\dstyle \int_{\mathbb{R}^n}\frac{1}{|B(x,r)|}
\int_{B(x,r)} \left(\frac{|x-z|}{|y-z|}\right)^{n\alpha}  dy \ d\nu(z).\\
\end{array}
\end{equation}

As discussed in the previous paragraph, $\frac{1}{|x|^{n\alpha}}$ is an $A_1$ weight on $\mathbb{R}^n$ when $0\leq \alpha<1$,
we have for any $x$, and $r>0$
\begin{equation}\dstyle \frac{\frac{1}{|B(x,r)|}
\dstyle \int_{B(x,r)} \left(\frac{1}{|y|^{n\alpha}}\right)  dy}{\dstyle \frac{1}{|x|^{n\alpha}}}\leq C,
\end{equation}
where $C$ is independent of $x$ or $r$. This is also true with the same constant $C$ when shifting by any $z\in \mathbb{R}^n$,

\begin{equation} \dstyle \frac{\frac{1}{|B(x,r)|}
\dstyle \int_{B(x,r)} \left(\frac{1}{|y-z|^{n\alpha}}\right)  dy}{\dstyle \frac{1}{|x-z|^{n\alpha}}}\leq C,
\end{equation}
Thus

\begin{equation}
\begin{array}{lcl}\label{77}
&&\dstyle \int_{\mathbb{R}^n}\frac{1}{|B(x,r)|}
\int_{B(x,r)} \left(\frac{|x-z|}{|y-z|}\right)^{n\alpha}  dy \ d\nu(z)\\
&\leq & \dstyle \int_{\mathbb{R}^n} C\ d\nu(z) = C,\\
\end{array}
\end{equation}
for any $r>0$ and $x\in \mathbb{R}^n$.
(\ref{eqnA1}) is established by plugging (\ref{77}) into (\ref{74}) and then (\ref{73}). This concludes the proof.
\end{proof}

\section{Proof of bi-Lipschitz parametrization}
In this section we are going to prove that under the same assumptions as in Theorem \ref{main}, the manifold
$(M^4,e^{2w}|dx|^2)$ is bi-Lipschitz to the Euclidean space. This gives a sufficient condition of bi-Lipschitz
to the Euclidean space. The argument to prove bi-Lipschitz parametrization from the existence of q.c map is standard. 
\begin{proof}of Theorem \ref{bilip}
The metric on $(M^4,e^{2w}|dx|^2)$ gives the distance function
$$d(x,y)=\inf_{\gamma} \int_{\gamma} e^{w(x)}ds,$$
where $ds$ is the differential of arc length, $\gamma$ connecting $x$ and $y$. Consider
the $\tilde{H'}$-q.c. map $f$ and the weight function $e^{4w}$.
For any curve $\gamma$ connecting $x$ and $y$,
\begin{equation}\label{5.1}\begin{array}{lcl}
\dstyle |f(x)-f(y)|&\leq &\dstyle |\int_{0}^{1}\frac{d}{dt}f(\gamma(t))dt|\\
&=&\dstyle |\int_{0}^1 Df(\gamma(t))\cdot \dot{\gamma}(t)dt|\\
&\leq&\dstyle  C\int_{\gamma}J_{f}^{\frac{1}{4}}(\gamma(s))ds\\
&\leq&C\dstyle  \int_{\gamma}e^w ds\\
&\leq &C (e^{4w}\mbox{-length of }\gamma).
\end{array}
\end{equation}
thus
$$|f(x)-f(y)|\leq C d(x,y).$$

\noindent On the other hand
\begin{equation}\begin{array}{lcl}\delta(x,y)&=&\dstyle \left(\int_{B_{x,y}}e^{4w(z)}dz\right)^{\frac{1}{4}}\\
&=&\omega(B_{x,y})^{\frac{1}{4}}\\
&\simeq &|f(x)-f(y)|,
\end{array}
\end{equation}
where $\simeq$ is because of the change of variable formula and a standard distortion theorem.
Combining the above two inequalities, we get
$$\delta(x,y)\leq C|f(x)-f(y)|\leq C d(x,y).$$
Also, (\ref{A}) says since $e^{4w}$ is a strong $A_\infty$ weight,
$$d(x,y)\leq C \delta(x,y).$$ Therefore
$$d(x,y)\simeq |f(x)-f(y)|$$
for all $x,y\in \mathbb{R}^4$ with implicit constant in $\simeq$ depending on $M^4$.
 Thus $f$ is a bi-Lipschitz map from $(M^4,e^{2w}|dx|^2)$ to Euclidean space $(\mathbb{R}^4, |dx|^2)$.
\end{proof}
\noindent This argument is true for all even dimension $n$.

Department of Mathematics, Princeton University, Princeton, NJ 08544, USA
E-mail Address: wangyi@math.princeton.edu

\end{document}